\documentclass{article}

\usepackage{PRIMEarxiv}
\usepackage{float}
\usepackage[utf8]{inputenc} 
\usepackage[T1]{fontenc}    
\usepackage{hyperref}       
\usepackage{url}            
\usepackage{booktabs}       
\usepackage{amsfonts}       
\usepackage{nicefrac}       
\usepackage{microtype}      
\usepackage{lipsum}
\usepackage{fancyhdr}       
\usepackage{graphicx}       
\graphicspath{{media/}}     
\usepackage{biblatex} 
\usepackage{amsmath}
\newtheorem{theorem}{Theorem}

\newtheorem{definition}{Definition}[section]
\newtheorem{proof}{Proof}[section]
\usepackage{bm}
  \usepackage{mathtools}
   \usepackage{amssymb}
   \DeclareUnicodeCharacter{230A}{\lfloor}
   \DeclareUnicodeCharacter{230B}{\lfloor}

\title{Amazing behavior and transition to chaos of some sequences using Collatz like problems and Quibic duffing
\thanks{\textit{\underline{Citation}}: 
\textbf{Authors. Title. Pages.... DOI:000000/11111.}} 
}

\author{
  Zeraoulia Rafik \\
  University of batna2.Algeria \\
  Departement of mathematics \\
  yabous ,khenchela\\
  \texttt{\{1\}r.zeraoulia@univ-batna2.dz} \\
}

\begin{document}
\maketitle

\begin{abstract}
In this paper we shall show amazing behavior of some discrete maps  using Collatze like problems and some advanced theories in analytic number theory and dynamical system,we have investigated the
driven cubic-quintic Duffing equation such that  we were
able to predict the number of limit cycles around the equilibrium and to develop a
theoretical approach to chaos suppression in damped driven systems using Collatze like problem sequences , some new results regarding behavior of that sequence are presented.
\end{abstract}


\section{Introduction}
The Collatz conjecture is one of the most famous unsolved problems in mathematics.\cite{13},\cite{7} The conjecture asks whether repeating two simple arithmetic operations will eventually transform every positive integer into 1. It concerns sequences of integers in which each term is obtained from the previous term as follows: if the previous term is even, the next term is one half of the previous term. If the previous term is odd, the next term is 3 times the previous term plus 1. The conjecture is that these sequences always reach 1, no matter which positive integer is chosen to start the sequence ,That  the conjecture has been tested via computers for numbers up to $\approx 5.48\cdot 10^{18}$, which is quite impressive, although we was thinking ahead of time it would have been tested by this time for incredibly large integers given we've had computers for  more than 50 years now.

Newly discovered fundamental theories (meta mathematics) of integer numbers may be used to formalise and formulate a new theoretical number system from which other formal analytical frameworks may be discovered (\cite{18}), primed and developed. The proposed number system (\cite{15}), as well as its most general framework which is based on the modelling results derived from an investigation of the Collatz conjecture (\cite{17}) (i.e., the 3x+1 problem), has emerged as an effective exploratory tool for visualising, mining and extracting new knowledge about quite a number of mathematical theorems and conjectures, including the Collatz conjecture  . an increased interest has been witnessed in studying the theory of discrete
dynamical systems including  Collatz , specifically of their associated difference equations. Sizable number of
works on the behavior and properties of pertaining solutions boundedness and
unboundedness of sequences which are derived From Collatz problems have been published in various areas of applied mathematics and physics. (\cite{1},\cite{3},\cite{4} ,\cite{14},\cite{5})

\begin{definition}
    The (borderline) Collatz-like problems: 
A map $f: \mathbb{N} \to \mathbb{N}$ will be called a Collatz-like map if 
\begin{equation}\label{1}
 0 \neq \lim_{n \to \infty} \left( \prod_{r=1}^n \frac{f(r)}{r} \right)^{1/n} \le 1 \end{equation}\ \ \ \   If the inequality  (\ref{1}) is an equality then the map $f$ will be called a borderline Collatz-like map.  
For each (borderline) Collatz-like map $f$, we have the  (borderline) Collatz-like problem asking whether its iterations diverges nowhere to infinity, i.e. $$\forall n>0, \ \exists m,r>0 \text{ with } f^{\circ (m+r)}(n) = f^{\circ m}(n).$$  
If the answer is yes, then let us call $f$ an acceptable (borderline) Collatz-like map.

This really  focus on a specific family of borderline Collatz-like problems:  
  
For any given $\alpha >0$, let us consider the following map \cite{10}
\begin{equation}\label{2}
f_{\alpha}: n \mapsto \left\{
    \begin{array}{ll}
         \left \lfloor{n\alpha} \right \rfloor & \text{ if } n \text{ even,}  \\
         \left \lfloor{n/\alpha} \right \rfloor & \text{ if } n \text{ odd.}
    \end{array}
\right.
\end{equation}

The map $f_{\alpha}$ in (\ref{2}) is borderline Collatz-like. Let $S$ be the set of $\alpha>0$ for which $f_{\alpha}$ is acceptable.

\end{definition}

One of the most important topic in analytic  number theory (\cite{12}) , which has attracted attention of
researchers in the field, is irrationality  measure of transcendental numbers like $\sqrt{2} ,\pi ,e ,\cdots$ ,(\cite{16},\cite{17}). In this paper we shall give amazing and surprising behavior of the following  discrete map which is deduced from  (\ref{2}) taking $\alpha=\sqrt{2}$ then ,let us consider \cite{22} :
$$f: n \mapsto \left\{
    \begin{array}{ll}
         \left \lfloor{n/\sqrt{2}} \right \rfloor & \text{ if } n \text{ even,}  \\
         \left \lfloor{n\sqrt{2}} \right \rfloor & \text{ if } n \text{ odd.}
    \end{array}
\right.$$

such that it involves $\sqrt{2},\pi$  and parity  and so on .

\section{Analysis and discussion}
Consider the following map:  

$$f: n \mapsto \left\{
    \begin{array}{ll}
         \left \lfloor{n/\sqrt{2}} \right \rfloor & \text{ if } n \text{ even,}  \\
         \left \lfloor{n\sqrt{2}} \right \rfloor & \text{ if } n \text{ odd.}
    \end{array}
\right.$$

Let $f^{\circ (r+1)}:=f \circ f^{\circ r}$, consider the orbit of $n=73$ for iterations of $f$, i.e. the sequence $f^{\circ r}(73)$:  $$73, 103, 145, 205, 289, 408, 288, 203, 287, 405, 572, 404, 285, 403, 569, 804, 568, 401, \dots$$  

It seems that this sequence diverges to infinity exponentially, and in particular, never reaches a cycle. Let illustrate that with the following picture of $(f^{\circ r}(73))^{1/r}$, with $200<r<20000$ .See Figure 1:

\begin{figure}[H]
    \centering
    \includegraphics[width=0.8\textwidth]{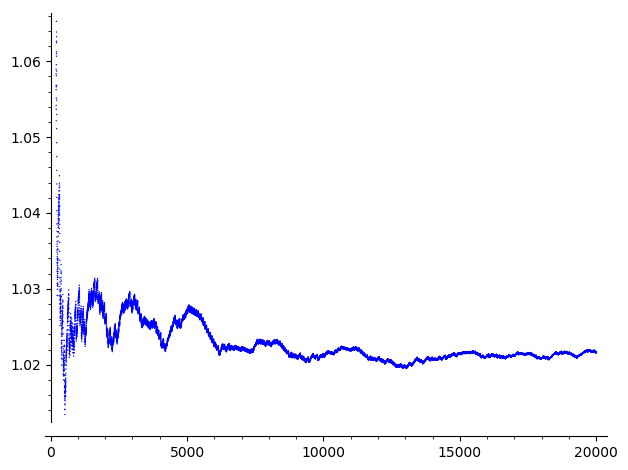}
    \caption{The iterated map for 73 times ($(f^{\circ r}(73))^{1/r}$, with $200<r<20000$)}
    \label{fig:my_label1}
\end{figure}
According to Figure1, it seems that $f^{\circ r}(73) \sim \delta^r$ with $\delta \sim 1.02$.

Now consider the probability (\cite{9}) of the $m$ first terms of the sequence $f^{\circ r}(73)$ to be even: $$p_{0}(m):= \frac{|\{ r<m \ | \  f^{\circ r}(73) \text{ is even}\}|}{m}.$$
Then $p_1(m):=1-p_0(m)$ is the probability of the $m$ first terms of $f^{\circ r}(73)$ to be odd.  

If we compute the values of $p_i(m)$ for $m=10^{\ell}$, $\ell=1,\dots, 5$, we get something  unexpected: 

It is unexpected because it seems that $p_0(m)$ does not converge to $1/2$, but to $\alpha \sim 0.465$.   
It matches with the above observation because $$ \delta \sim 1.02 \sim \sqrt{2}^{(0.535-0.465)} = \sqrt{2}^{(1-2 \times 0.465)} \sim \sqrt{2}^{(1-2\alpha)}.$$
\bigskip

$$
\begin{matrix}
\ell  & p_0(10^{\ell}) &p_1(10^{\ell}) \\
1 &0.2&0.8 \\
2 &0.45&0.55 \\
3 &0.467&0.533 \\
4 &0.4700&0.5300 \\
5 &0.46410&0.53590 \\
6 & 0.465476& 0.534524
\end{matrix}
$$

The line for $\ell = 6$ was computed Using Pari/GP and setting internal precision to $15000$ decimal digits we were able to get $p_1(10^6) = 0.534524$,The value of $f^{1e6}(73)$ is about $3.89439e10394$, (thus one needs such a big precision) and the $\log_{73}()$ of it is $5578.52$.
Now one can ask is it true that $f^{\circ r}(73)$ never reach a cycle, that $(f^{\circ r}(73))^{1/r}$ converges to $\delta \sim 1.02$, that $p_0(m)$ converges to $\alpha \sim 0.465$, and that $\delta^24^{\alpha} = 2$?  What are the exact values of $\delta$ and $\alpha$?  namely, better approximations?
The following Figure  provides the values of $p_0(m)$ for $100 < m < 20000$,See Figure2

\begin{figure}[H]
    \centering
    \includegraphics[width=0.7\textwidth]{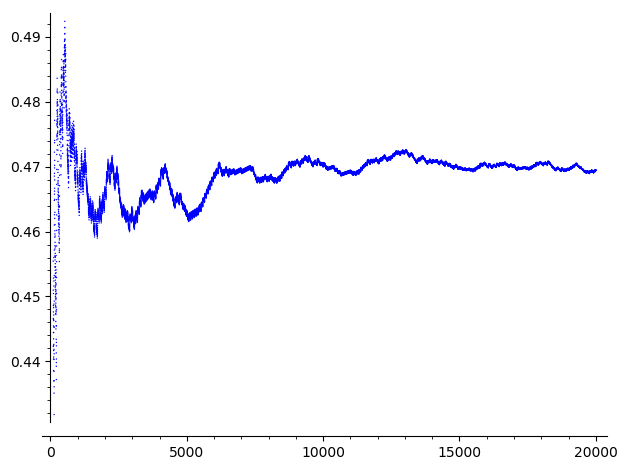}
    \caption{ values of $p_0(m)$ for $100 < m < 20000$}
    \label{fig:my_label3}
\end{figure}
Note that this phenomenon is not specific to $n=73$, but seems to happen as frequently as $n$ is big, and then, the analogous probability seems to converge to the same $\alpha$. If $n <100$, then it happens for $n=73$ only, but for $n<200$, it happens for $n=73, 103, 104, 105, 107, 141, 145, 146, 147, $ $ 148,  149, 151, 152, 153, 155, 161, 175, 199$; and for $10000 \le n < 11000$, to exactly $954$ ones.

Below is the picture shown by Figure3 as Figure2 but for $n=123456789$:
\begin{figure}[H]
    \centering
    \includegraphics[width=0.8\textwidth]{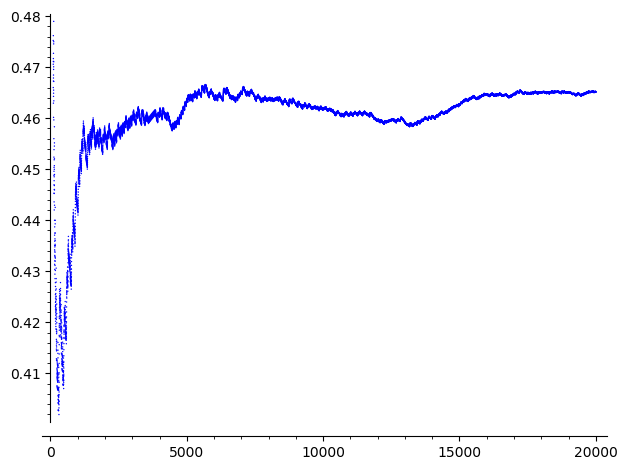}
    \caption{ values of $p_0(m)$ for $100 < m < 20000$, $n=123456789$}
    \label{fig:my_label4}
\end{figure}
One can ask  Is it true that the set of $n$ for which the above phenomenon happens has natural density one? Is it cofinite? When it happens, does it involves the same constant $\alpha$?

There are exactly $1535$ numbers $n<10000$ for which the above phenomenon does not  happen. The next Figure,namely ,Figure 4 displays for such $n$ the minimal $m$ (in blue) such that $f^{\circ m}(n) = f^{\circ (m+r)}(n)$ for some $r>0$, together with the miniman such $r$ (in red):
\begin{figure}[H]
    \centering
    \includegraphics[width=0.7\textwidth]{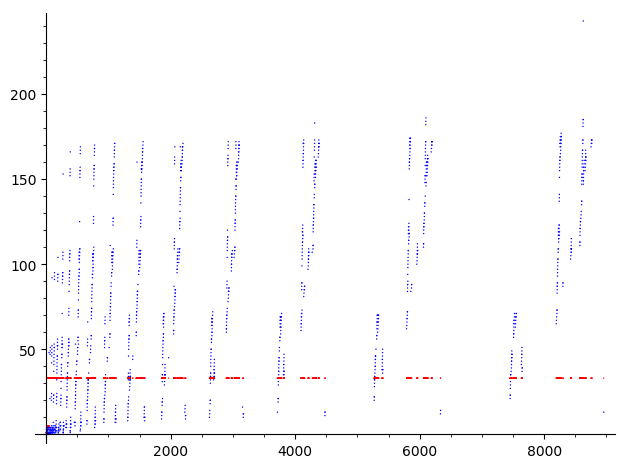}
    \caption{ The minimal $m$ for which $f^{\circ m}(n) = f^{\circ (m+r)}(n)$ for some $r>0$ }
    \label{fig:my_label5}
\end{figure}

In fact all these numbers (as first terms) reach the following cycle of length $33$:  

$$(15,21,29,41,57,80,56,39,55,77,108,76,53,74,52,36,25,35,49,69,97,137,193,272,192,135,190,134,94,66,46,32,22)$$   

except the following ones: $$7, 8, 9, 10, 12, 13, 14, 18, 19, 20, 26, 27, 28, 38, 40, 54,$$ which reach $(5,7,9,12,8)$, and that ones $1, 2, 3, 4, 6$ which reach $(1)$, and $f(0)=0$.

If the pattern continues like above up to infinity, they must have infinity many such $n$.  We may need to ask if  there infinitely many $n$ reaching a cycle? Do they all reach the above cycle of length $33$ (except the
few ones mentioned above)? What is the formula of these numbers $n$?  

Below in Figure5 is their counting function (it looks logarithmic):   
\begin{figure}[H]
    \centering
    \includegraphics[width=0.7\textwidth]{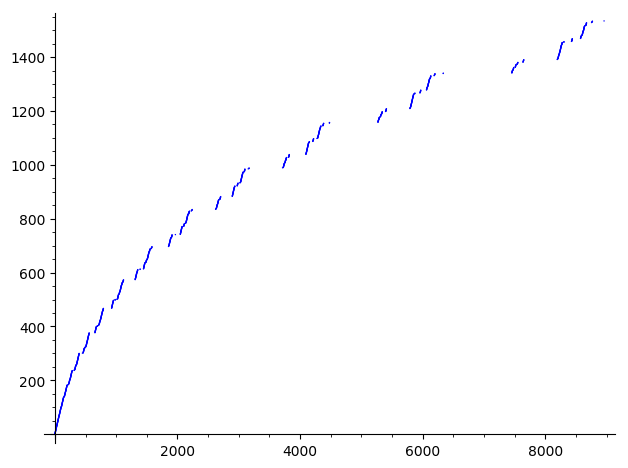}
    \caption{ counting function of the happned phenomena}
    \label{fig:my_label6}
\end{figure}

\section{Main result}
\begin{itemize}
\item 1)
A number $ m$  admits no predecessor iff the interval $[m\sqrt{2},(m+1)\sqrt{2}]$ admits no even number and the interval $[m/\sqrt{2},(m+1)/\sqrt{2}]$ admits no odd number. There are exactly $r_{\ell}$ such numbers $m<10^{\ell}$ with $\ell=1,2,3,4,5,6$ and $r_{\ell}=2,29,292,2928,29289,292893$ Strangely, for $\ell \le 6$ we observe that $r_{\ell-1} = ⌊ r_{\ell}/10⌋$
\item 2) Numbers without predecessors those of the form  $\lfloor n(2+\sqrt2)\rfloor$
Moreover , numbers with one predecessor are those of the form $\lfloor2k\sqrt2\rfloor$ and numbers with two predecessors those of the form $\lfloor(2k-1)\sqrt2\rfloor$

\item3)The homoclinic orbit of unperturbed system for Quibic duffing oscillators using Collatz sequences  separates the phase plane into two areas.
Inside the separatrix curve the orbits are around one of the centers, and
outside the separatrix curve the orbits surround both the centers and the
saddle point
\end{itemize}
\section{Analysis of the first result}
We  take pairs of $(m,n)$ (\cite{2}) for consecutive $m$ and their 1-step predecessors $n$ such that $f(n)=m$. The value $n=0$ indicates, that $m$ has no predecessor. I didn't reflect, that one $m$ can have two predecessors, but if $n/2$ is odd, then $n/2$ is a second predecessor (\cite{7}).(This makes the table more interesting, because all odd predecessors $n$ are overwritten by the even predecessors $2n$...     
   
Moreover, a nearly periodic structure occurs. We tried to resemble this by the arrangement of three or four columns of $(m,n)$ such that the first column contains all $m$ which have no predecessor. The basic pattern is not really periodic, but has super-patterns which again seem to be periodic but actually aren't. This pattern-superpattern-structure is also recursive. It reminds me of a similar structure when I looked at $\beta=\log_2(3)$ and  found a similar style of pattern-superpattern-supersuperpattern-... and is there related to the continued fraction of $\beta$.     
So We think we'll get no nice description for the cases $m$ which have no predecessor
\begin{figure}[H]
    \centering
    \includegraphics[width=0.7\textwidth]{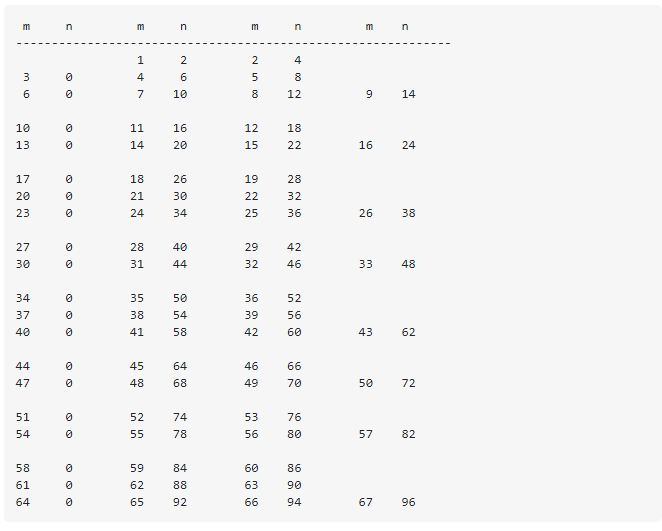}
    \caption{ Interpretation table for predecessors of the happened phenomena}
    \label{fig:my_label7}
\end{figure}
 Some more explanation on the idea of "recursive aperiodic pattern". 
If we list the values $m$ which have no predecessor, we get       
\begin{verbatim}
m_k:   3, 6,10,13, 17,20,23,27,30,... 
\end{verbatim}

Writing the differences (We have prepended a zero-value to the above list of $m_k$)         
\begin{verbatim}
,3,3,4  ,3,4  ,3,3,4  ,3,4  ,3,3,4  ,3,4  ,3,4  ,3,3,4 , ...      
\end{verbatim}

We note, that we have a pattern of two different words: `3,3,4` and `3,4` repeating, but aperiodical. Let's denote the longer one with the capital `A` and the shorter one with the small `a` (and `A` means a difference of 10 and `a` of 7).      
We get 
\begin{verbatim}
     Aa Aa Aaa 
     Aa Aaa 
     Aa Aa Aaa 
     Aa Aaa
     Aa Aa Aaa
     Aa Aaa
     Aa ...       
\end{verbatim}

Again we find only two kind of "words". Let's them shorten by `Aaa`=`B` and `Aa`=`b`. `B` means now a difference of 24, `b` of 17. 
Then we get   
\begin{verbatim}

       bbB bB
       bbB bB
       bbB bB bB
       bbB bB
       bbB bB bB
       bbB bB
       bbB bB
       bbB bB bB
       ... 
\end{verbatim}

Next obvious step gives     
\begin{verbatim}
       Cc Cc Ccc
       Cc Ccc
       Cc Cc Ccc
       Cc Ccc
       Cc Cc Ccc
       Cc Ccc
       ... 
\end{verbatim}

with `c` representing a difference of$ 17+24=41$ and  $C $of $17+17+24=58$.       
And so on.    
If We recall correctly, then with the mentioned case of working with $\beta = \log_2(3)$ the same style of recursive pattern reflected the convergents of the continued fractions of $\beta$.       
The first few differences here match the convergents of the continued fraction of $\sqrt2$ so far: 
\begin{verbatim}

                a    b    c                    short patterns
     -------------------------------------
    [1  1  3    7    17   41  99   239   577  ...  ]  convergents of contfrac(sqrt(2))
    [0  1  2    5    12   29  70   169   408  ...  ] 
     -------------------------------------...
              A/2   B/2   C/2                  long patterns              

\end{verbatim}

 The above can be explained by the following:
\begin{itemize}
\item1) a number of the form $\lfloor2k\sqrt2\rfloor$ has exactly one predecessor $4k$;
\item2)a number of the form $\lfloor(2k-1)\sqrt2\rfloor$ has exactly two predecessors $2k-1$ and $4k-2$;
\item3) a number has no predecessors iff it has form $\lfloor n(2+\sqrt2)\rfloor$.
\end{itemize}

 Using a back-step algorithm (recursive) it seems we've got the predecessing tree of $m=73$. If no bugs, then this tree would also be complete. *(But my routine may still be buggy, please check the results!)          

The back-steps go from top-right south-west (antidiagonal) downwards. When there are two possible predecessors, they occur in the same column, but on separate rows.     
If there is a predecessor without further predecessor, a short line (`---`) is printed.
\begin{verbatim}


                                        73   <--- start
                                    104   
                                148   
                            105 ---
  
                            210   
                        149   
                    212   
                300 ---
  
                        298   
                    211 ---
  
                    422   
                299   
            424   
        600 ---
  
                598   
            423 ---
  
            846 ---
        ---------------------------- tree seems to be complete (please check for errors!)
\end{verbatim}

We may give heuristic proof for our main results using Beatty theorem \cite{19} which it states that:
\begin{theorem}
    given an irrational number $r>1$  there exists  $s>1$ so  that the Beatty sequences $\bm{B_r}$ and $\bm{B_s}$ partition the set of positive integers: each positive integer belongs to exactly one of the two sequences
\end{theorem}

\begin{proof}
    We can say with Beatty theorem(Theorem 1) : $A=\{E(n(\sqrt{2}+2)) \text{ ; } n\in\mathbb N^*\}$ and $B=\{E(n \sqrt{2});n\in\mathbb N^*\}$ is a partition of $\mathbb N^*$
And we have $E(n(\sqrt{2}+2))=2n+E(n\sqrt{2})$
with $E$ is the function integer part ,and this proves the partial of result 2 (Form of numbers without predecessor).

For the First result , let assume that the probabilty for an integer $n$ to be odd is $\frac{1}{2}$, and that the probabilty for $f(n)$ to be odd when $n$ is even (resp. odd) is also $\frac{1}{2}$. We will observe that (surprisingly) it is no more $\frac{1}{2}$ for $f^{\circ r}(n)$ when $r \ge 2$ (in some sense, the probability does not commute with the composition of $f$ with itself).  
\begin{itemize}
\item1) if $n$ and $m=f(n)$ are even: note that $\frac{n}{\sqrt{2}} = m+\theta$ (with $0 < \theta < 1$) so that $m=\frac{n}{\sqrt{2}}- \theta$, then $$f^{\circ 2}(n) = f(m) = \left \lfloor{\frac{m}{\sqrt{2}}} \right \rfloor = \left \lfloor{\frac{\frac{n}{\sqrt{2}}- \theta}{\sqrt{2}}} \right \rfloor = \left \lfloor \frac{n}{2} - \frac{\theta}{\sqrt{2}}\right \rfloor$$ but $\frac{n}{2}$ is even with probability $\frac{1}{2}$, so in this case, $f^{\circ 2}(n)$ is odd with probability $\frac{1}{2}$.  

\item2) if $n$ is even and $m=f(n)$ is odd: $$f^{\circ 2}(n) = f(m) = \left \lfloor\sqrt{2}m \right \rfloor = \left \lfloor \sqrt{2}(\frac{n}{\sqrt{2}} - \theta) \right \rfloor = \left \lfloor n - \sqrt{2} \theta) \right \rfloor$$ but $n$ is even and the probability for $0<\sqrt{2} \theta<1$ is $\frac{\sqrt{2}}{2}$ (because $\theta$ is assumed statistically equidistributed \cite{9} on the open interval $(0,1)$), so $f^{\circ 2}(n)$ is odd with probability 
$\frac{\sqrt{2}}{2}$.    

\item3) if $n$ is odd and $m=f(n)$ is even:  
$$f^{\circ 2}(n) = f(m) = \left \lfloor{\frac{m}{\sqrt{2}}} \right \rfloor = \left \lfloor{\frac{\sqrt{2}n-\theta}{\sqrt{2}}} \right \rfloor = \left \lfloor n - \frac{\theta}{\sqrt{2}} \right \rfloor  $$
but $n$ is odd and $0 < \frac{\theta}{\sqrt{2}}<1$, so $f^{\circ 2}(n)$ is even.  

\item4) if $n$ is odd and $m=f(n)$ is odd:    
$$f^{\circ 2}(n) = f(m) = \left \lfloor \sqrt{2} m \right \rfloor = \left \lfloor \sqrt{2} (\sqrt{2}n-\theta) \right \rfloor = \left \lfloor 2n - \sqrt{2} \theta \right \rfloor $$
but $2n$ is even and the probability for $0<\sqrt{2} \theta<1$ is $\frac{\sqrt{2}}{2}$, so $f^{\circ 2}(n)$ is odd with probability $\frac{\sqrt{2}}{2}$.  
\end{itemize}

By combining these four cases together, we deduce that the probability for $f^{\circ 2}(n)$ to be odd is $$\frac{1}{2} \times \frac{1}{2} \times (\frac{1}{2} + \frac{\sqrt{2}}{2} + 0 + \frac{\sqrt{2}}{2}) = \frac{2\sqrt{2}+1}{8}$$  

By continuing in the same way, we get that the probability for $f^{\circ 3}(n)$ to be odd is:  
$$ \frac{1}{4} (\frac{1}{2}\frac{1}{2} + \frac{1}{2}\frac{\sqrt{2}}{2} + \frac{\sqrt{2}}{2}\frac{\sqrt{2}}{2} + 1\frac{1}{2} + \frac{\sqrt{2}}{2}\frac{\sqrt{2}}{2})  = \frac{\sqrt{2}+7}{16}$$   

For $2 \le r \le 24$, we computed the probability $p_r$ for $f^{\circ r}(n)$ to be odd (see Appendix). It seems (experimentally) that $p_r$ converges to a number $\simeq 0.532288725 \simeq \frac{8+3\sqrt{2}}{23}$ by Inverse Symbolic Calculator . This leads to the following question/conjecture:    
$$\lim_{r \to \infty}p_r = \frac{8+3\sqrt{2}}{23} \ \  ?$$  

If so, consider the number $\alpha$ , then $$\alpha = 1-\frac{8+3\sqrt{2}}{23} = \frac{15-3\sqrt{2}}{23} \simeq 0.467711,$$ which matches with above   computation (Analysis and discussion section). And next, we would have:  
$$ \delta = \frac{\sqrt{2}}{2^{\alpha}}= 2^{\frac{1}{2}-\alpha} = 2^{\frac{6\sqrt{2}-7}{46}} \simeq 1.022633$$
\end{proof}

\section{Homoclinic orbits and chaos in the un-perturbed system using Collatz problem}

Although no universally accepted mathematical definition of chaos
exists, a commonly used definition originally formulated by Robert L.
Devaney says that, to classify a dynamical system as chaotic, it must
have these properties:
\begin{itemize}
    \item 1) it must be sensitive to initial conditions
\item2) it must be topologically mixing
\item3) it must have dense periodic orbits
\end{itemize}

Firstly we may consider  the harmonically driven damped pendulum which is often used as a simple example of
a chaotic system, the equation is just 
\begin{equation}
\ddot{\phi}+\frac{1}{q}\dot{\phi}+\sin \phi =A\cos (\omega t)  \label{c_1}
\end{equation}%
As long as $A$ and $\omega $ are small it behaves like a driven harmonic
oscillator, and asymptotically settles into regular oscillations with a
fixed period. However, as $A$ (or $\omega $) are increased, with the rest of
parameters fixed, the system undergoes a cascade of period doubling
bifurcations leading to chaotic behavior, which then gives way to regular
oscillations again when it is increased further. For example, when $q=2$ and 
$\omega =2/3$ the first period doubling ("symmetry breaking") occurs at $%
A\approx 1.07$ and the first chaos at $A\approx 1.08$. These rigorous
results seem to be obtained by numerical simulations. One can be actually
interested in situations where chaos does not occur \cite{20}. Are there known
rigorous conditions on $A,\omega $ and $q$ that put the system below the
first period doubling? \ However this question does not belong to the aim of
this paper but it would be very interesting to conclude somethings about
chaotics behaviors of some dynamics and to discover new ways to supress
chaos in the cubic-Quintic Duffing Equation using some discrete iterated map  which it is the aim of our
research in this paper.
For the unperturbed system with fractional order displacement, when $%
\varepsilon =0$, the differential equation (\ref{c_1}) can be reformulated as For $A=0$, to%
\begin{equation}
\ddot{x}-ax+bx^{3}+cx^{5}=0.~  \label{r1}
\end{equation}%
Let 
\begin{equation}
\Delta :=b^{2}+4ac.  \label{r2}
\end{equation}%
Equilibrium points for $\Delta >0$ are : 
\begin{equation}
\left( x,\dot{x}\right) =\left( \pm \sqrt{\frac{-b+\sqrt{b^{2}+4ac}}{2c}}%
,0\right) \text{ : centers}  \label{r3}
\end{equation}%
Define%
\begin{equation}
x_{e}^{+}=\sqrt{\frac{-b+\sqrt{b^{2}+4ac}}{2c}}\text{ and }x_{e}^{-}=-\sqrt{%
\frac{-b+\sqrt{b^{2}+4ac}}{2c}}  \label{r3a}
\end{equation}%
The energy function for (\ref{r1}) \ is 
\begin{equation}
\frac{1}{2}\dot{x}(t)^{2}-\frac{1}{2}ax(t)^{2}+\frac{1}{4}bx(t)^{4}+\frac{1}{%
6}cx(t)^{6}=K  \label{r4}
\end{equation}%
where $K$ is the energy constant dependent on the initial amplitude $%
x(0)=x_{0}$ and initial velocity $x^{\prime }(0)=\dot{x}_{0}$ : 
\begin{equation}
K=\frac{1}{2}\dot{x}_{0}^{2}-\frac{1}{2}ax_{0}^{2}+\frac{1}{4}bx_{0}^{4}+%
\frac{1}{6}cx_{0}^{6}.  \label{r5}
\end{equation}%
Dependently on $K$, the level sets are different. For all of them it is
common that they form closed periodic orbits which surround the fixed points 
$(\mathrm{x},\dot{\mathrm{x}})=(x_{e}^{+},\ 0)$ or $(\mathrm{x},\dot{\mathrm{%
x}})=(x_{e}^{-},\ 0)$ or all the three fixed points $(x_{e}^{\pm },\ 0)$ and 
$(0,\ 0)$. The boundary between these two groups of orbits corresponds to $%
K=0$, when%
\begin{equation}
\dot{x}_{0}=\pm x_{0}\sqrt{\frac{1}{6}\left(
6a-3bx_{0}^{2}-2cx_{0}^{4}\right) }.  \label{r5a}
\end{equation}%
The level set 
\begin{equation}
\frac{\dot{x}^{2}}{2}-\frac{1}{2}ax^{2}+\frac{1}{4}bx^{4}+\frac{1}{6}cx^{6}=0%
\text{ }\quad   \label{r5b}
\end{equation}%
is composed of two homoclinic orbits 
\begin{equation}
\Gamma _{+}^{0}(\mathrm{t})\equiv (x_{+}^{0}(\mathrm{t}),\dot{x}_{+}^{0}(%
\mathrm{t})),\text{ }\quad   \label{r5b1}
\end{equation}%
\begin{equation}
\Gamma _{-}^{0}(\mathrm{t})\equiv (x_{-}^{0}(\mathrm{t}),\dot{x}_{-}^{0}(%
\mathrm{t})),\text{ }  \label{r5b2}
\end{equation}%
which connect the fixed hyperbolic saddle point $(0,\ 0)$ to itself and
contain the stable and unstable manifolds. The functions $x_{\pm }^{0}(%
\mathrm{t})$ may be evaluated using these two  formulas
\begin{equation}
M(t_{0})=\int_{-\infty }^{+\infty }\dot{x}^{0}(t)[\gamma \cos \omega
(t+t_{0})-\delta \dot{x}^{0}(t)]dt,\text{ }\quad   \label{M5}
\end{equation}%
with 
\begin{equation}
x^{0}(t)=\frac{A\text{sech}\left( \sqrt{k}t\right) }{\sqrt{1+\lambda \cdot 
\text{sech}^{2}\left( \sqrt{k}t\right) }}.  \label{D6}
\end{equation}%
. See
Figure 5
\begin{figure}[H]
    \centering
    \includegraphics[width=0.7\textwidth]{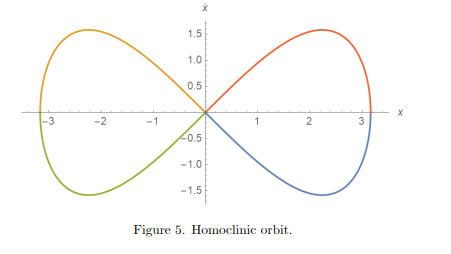}
    \label{fig:Alvaro5}
\end{figure}

\bigskip The homoclinic orbit \cite{21} separates the phase plane into two areas.
Inside the separatrix curve the orbits are around one of the centers, and
outside the separatrix curve the orbits surround both the centers and the
saddle point. Physically it means that for certain initial conditions the
oscillations are around one steady-state position \cite{23}, and for others around all
the steady- state solutions (two stable and an unstable).

Now ,in the second case which uses the iterated map (Collatze like problems sequences ) as the RHS  of equation (\ref{c_1}) ,that is perturbed system, We have noted transition to chaos(sensitivity to initial condition) ,Let us consider the following IVP (intial value problem):
\begin{equation}\label{CHAOS}
\ddot{x}-ax+bx^{3}+cx^{5}=f,\dot(x(0))=0,x(0)=0
\end{equation}%
such that :

$$f: n \mapsto \left\{
    \begin{array}{ll}
         \left \lfloor{n/\sqrt{2}} \right \rfloor & \text{ if } n \text{ even,}  \\
         \left \lfloor{n\sqrt{2}} \right \rfloor & \text{ if } n \text{ odd.}
    \end{array}
\right.$$
Let $f^{\circ (r+1)}:=f \circ f^{\circ r}$, consider the orbit of $n=73$ for iterations of $f$, i.e. the sequence $f^{\circ r}(73)$:  $$73, 103, 145, 205, 289, 408, 288, 203, 287, 405, 572, 404, 285, 403, 569, 804, 568, 401, \dots$$  (see 
It seems that this sequence diverges to infinity exponentially, and in particular, never reaches a cycle. Let illustrate that with the following picture of $(f^{\circ r}(73))^{1/r}$, with $200<r<20000$ , See Figure 7

\begin{figure}[H]
    \centering
    \includegraphics[width=0.7\textwidth]{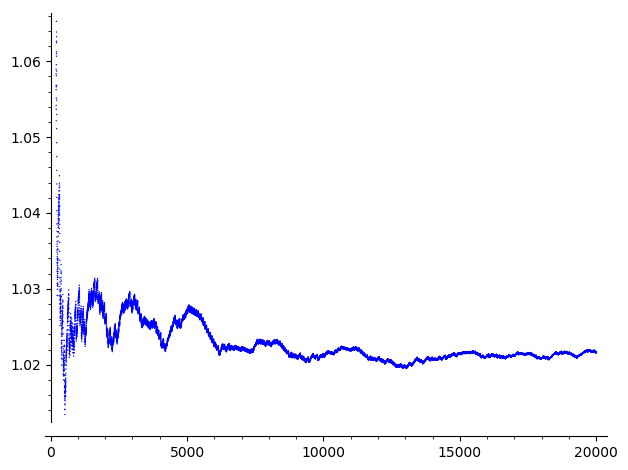}
    \caption{ expenential divergence for $200<r<20000$ }
    \label{fig:my_label8}
\end{figure}

And for initial sensitivity of IVP  corresponding to the system (\ref{CHAOS}) ,we have obtained the following Figure up to $r=20000$ (small perturbation),see Figure 8:

\begin{figure}[H]
    \centering
    \includegraphics[width=0.7\textwidth]{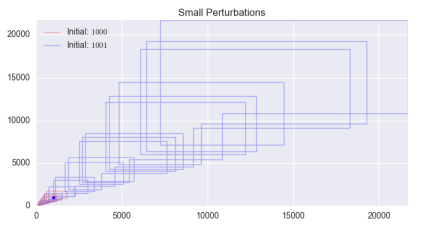}
    \caption{ Sensitivity to initial condition for small perturbation up to $r=20000$ for the Collatze sequence}
    \label{fig:my_label9}
\end{figure}
We may now show a phase space plot of the trajectory up to $r=20000$,namely, small perturbation ,See Figure9
\begin{figure}[H]
    \centering
    \includegraphics[width=0.7\textwidth]{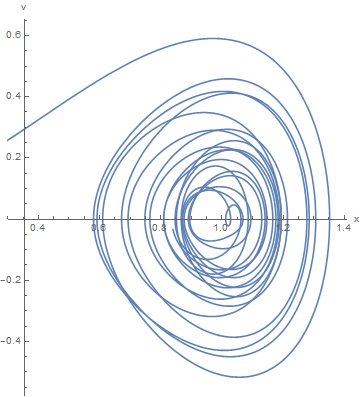}
    \caption{ transition to chaos $r=20000$ for (\ref{CHAOS}) using Collatze sequence for 73 iteration}
    \label{fig:my_labelp}
\end{figure}
This looks complicated, but in fact, most of the plot shows the initial period of time during which the motion is approaching its final behavior which is much simpler. The early behavior is called an "initial transient".

\section*{Acknowledgments}
This work is based on  work of  Sebastien Palcoux for his Arxiv research paper entitled "  unexpected behavior of some transcendental number like $\sqrt{2}$ " \cite{22}

\section{declaration statement}
The authors declare that they have no known competing financial interests or personal relationships that could have appeared to influence the work reported in this paper.

\section{Data Availability}
The authors were unable to find a valid data repository for the data used in this study

\section{Appendix}
Computation:
\begin{verbatim}
sage: for i in range(3,26):
....:     print(sq2(i))
....:
[1/4*sqrt(2) + 1/8, 0.478553390593274]
[1/16*sqrt(2) + 7/16, 0.525888347648318]
[3/32*sqrt(2) + 13/32, 0.538832521472478]
[15/64*sqrt(2) + 13/64, 0.534581303681194]
[5/128*sqrt(2) + 61/128, 0.531805217280199]
[39/256*sqrt(2) + 81/256, 0.531852847392776]
[93/512*sqrt(2) + 141/512, 0.532269260352925]
[51/1024*sqrt(2) + 473/1024, 0.532348527032254]
[377/2048*sqrt(2) + 557/2048, 0.532303961432938]
[551/4096*sqrt(2) + 1401/4096, 0.532283123258685]
[653/8192*sqrt(2) + 3437/8192, 0.532285334012406]
[3083/16384*sqrt(2) + 4361/16384, 0.532288843554459]
[3409/32768*sqrt(2) + 12621/32768, 0.532289246647030]
[7407/65536*sqrt(2) + 24409/65536, 0.532288816169701]
[22805/131072*sqrt(2) + 37517/131072, 0.532288667983386]
[24307/262144*sqrt(2) + 105161/262144, 0.532288700334941]
[72761/524288*sqrt(2) + 176173/524288, 0.532288728736551]
[159959/1048576*sqrt(2) + 331929/1048576, 0.532288729880941]
[202621/2097152*sqrt(2) + 829741/2097152, 0.532288725958633]
[639131/4194304*sqrt(2) + 1328713/4194304, 0.532288724978704]
[1114081/8388608*sqrt(2) + 2889613/8388608, 0.532288725350163]
[1825983/16777216*sqrt(2) + 6347993/16777216, 0.532288725570602]
[5183461/33554432*sqrt(2) + 10530125/33554432, 0.532288725561857]

\end{verbatim}

\textbf{Code:}
\begin{verbatim}
def sq2(n):
    c=0
    for i in range(2^n):
        l=list(Integer(i).digits(base=2,padto=n))
        if l[-1]==1:
            cc=1/4
            for j in range(n-2):
                ll=[l[j],l[j+1],l[j+2]]
                if ll==[0,0,0]:
                    cc*=1/2
                if ll==[0,0,1]:
                    cc*=1/2
                if ll==[0,1,0]:
                    cc*=(1-sqrt(2)/2)
                if ll==[0,1,1]:
                    cc*=sqrt(2)/2
                if ll==[1,0,0]:
                    cc*=1
                if ll==[1,0,1]:
                    cc=0
                    break
                if ll==[1,1,0]:
                    cc*=(1-sqrt(2)/2)
                if ll==[1,1,1]:
                    cc*=sqrt(2)/2
            c+=cc
    return [c.expand(),c.n()]
\end{verbatim}
\begin{equation}
M(t_{0})=\int_{-\infty }^{+\infty }\dot{x}^{0}(t)[\gamma \cos \omega
(t+t_{0})-\delta \dot{x}^{0}(t)]dt,\text{ }\quad   \label{c5}
\end{equation}%
Let 
\begin{equation}
x^{0}(t)=\frac{A\text{sech}\left( \sqrt{k}t\right) }{\sqrt{1+\lambda \cdot 
\text{sech}^{2}\left( \sqrt{k}t\right) }}.  \label{c6}
\end{equation}%


\end{document}